\title{Avoider-Enforcer games played on edge disjoint hypergraphs}
\author{{Asaf Ferber \thanks{School of Mathematical Sciences,
Raymond and Beverly Sackler Faculty of Exact Sciences, Tel Aviv
University, Tel Aviv, 69978, Israel. Email:
ferberas@post.tau.ac.il}} \quad{Michael Krivelevich\thanks{School of
Mathematical Sciences, Raymond and Beverly Sackler Faculty of Exact
Sciences, Tel Aviv University, Tel Aviv, 69978, Israel. Email:
krivelev@post.tau.ac.il. Research supported in part by USA-Israel
BSF Grant 2010115 and by grant 912/12 from the Israel Science
Foundation.}} \quad{Alon Naor
\thanks{School of Mathematical Sciences, Raymond and Beverly Sackler
Faculty of Exact Sciences, Tel Aviv University, Tel Aviv, 69978,
Israel. Email: alonnaor@post.tau.ac.il}}  }

\documentclass[11pt]{article}
\usepackage{amsmath,amssymb,latexsym,color,epsfig,a4,enumerate}

\newif\ifnotesw\noteswtrue

\def\({\left(}
\def\){\right)}

\def\cF{{\cal F}}

\def\cF{{\cal F}}

\def\cS{{\cal S}}

\newtheorem{theorem}{Theorem}[section]
\newtheorem{lemma}[theorem]{Lemma}

\newtheorem{corollary}[theorem]{Corollary}

\def\cM{{\cal M}}

\renewcommand{\epsilon}{\varepsilon}

\newenvironment{proof}{\noindent{\bf Proof\,}}{\hfill$\Box$}

\begin{document}
\maketitle

\begin{abstract}
We analyze Avoider-Enforcer games played on edge disjoint
hypergraphs, providing an analog of the classic and well known game
$Box$, due to Chv\'{a}tal and Erd\H{o}s. We consider both strict and
monotone versions of Avoider-Enforcer games, and for each version we
give a sufficient condition to win for each player. We also present
applications of our results to several general Avoider-Enforcer
games.
\end{abstract}

\section{Introduction}

Let $p$ and $q$ be two positive integers, let $X$ be a set and let
$\cF \subseteq 2^X$ be a family of \emph{target sets}. In a $(p,q)$
Avoider-Enforcer game $\cF$ two players, called Avoider and
Enforcer, alternately claim $p$ and $q$ previously unclaimed
elements of the \emph{board} $X$ per move, respectively. If the
number of unclaimed elements is strictly less than $p$ (respectively
$q$) before Avoider's (respectively Enforcer's) move, then he claims
all these elements. The definition of the game is complete by
stating which player begins the game. The game ends when all the
elements of the board have been claimed. Avoider loses the game if
by the end of the game he has claimed all the elements of some
target set. Otherwise, Avoider wins.

Avoider-Enforcer games are the mis\`{e}re version of the well
studied Maker-Breaker games. In a $(p,q)$ Maker-Breaker game $\cF$
two players, called Maker and Breaker, alternately claim $p$ and $q$
previously unclaimed elements of the board $X$ per move,
respectively. Maker wins if by the end of the game he has claimed
all the elements of some $F\in \cF$.

It turns out that Avoider-Enforcer games are much harder to analyze
than Maker-Breaker games. One of the main reasons for this
difficulty is the lack of bias monotonicity in this type of games
(see e.g. \cite{HKSS1,HKSS2}). While in a Maker-Breaker game it is
never a disadvantage to claim more elements per move (for either of
the players), it is sometimes a disadvantage to claim less elements
per move in an Avoider-Enforcer game (for either of the players).

In order to overcome this difficulty, Hefetz et al. proposed a bias
monotone version for Avoider-Enforcer games \cite{HKSS2}. In this
version Avoider and Enforcer claim \textbf{at least} $p$ and $q$
board elements per move, respectively. Throughout the paper we refer
to this set of rules as the \emph{monotone} rules, as opposed to the
\emph{strict} rules, and to the games played by each set of rules as
\emph{monotone} and \emph{strict} games, respectively. It is worth
mentioning that these seemingly minor adjustments in the rules may
completely change the game. For example, even in such a natural game
as \emph{the connectivity game} -- where the board is $E(K_n)$ and
Avoider's goal is to avoid having a spanning connected graph -- the
two versions of the game are essentially different. In the strict
rules, Avoider wins the $(1,q)$ game if and only if at the end of
the game he has at most $n-2$ edges \cite{HKSS1} (i.e. $q \geq
\lfloor\frac{n}{2}\rfloor$ or $q \geq \lfloor\frac{n}{2}\rfloor-1$,
depending on the parity of $n$ and the identity of the first
player). On the other hand, the asymptotic threshold for the
property ``Avoider wins the $(1,q)$ connectivity game played on
$E(K_n)$ according to the monotone rules" is $\frac{n}{\ln n}$ (see
\cite{HKSS2,KS}).

One of the main tools in analyzing Avoider-Enforcer games is the
following sufficient condition for Avoider's win which was proved by
Hefetz et al. \cite{HKSS1}, and is motivated by the generalized
Erd\H{o}s-Selfridge's sufficient condition for Breaker's win due to
Beck (see \cite{Beck,ES}):

\begin{theorem} \label{ReverseBox::AvWinCrt} [Theorem 1.1 \cite{HKSS1}]
If Avoider is the last player (i.e., the player to make the last
move) and $$\sum_{F\in
\cF}\left(1+1/p\right)^{-|F|}<\left(1+1/p\right)^{-p}$$ then Avoider
wins the $(p,q)$ game $\cF$ for every $q\geq 1$.

If Enforcer is the last player then the above sufficient condition
can be relaxed to $$\sum_{F\in \cF}\left(1+1/p\right)^{-|F|}<1.$$
\end{theorem}

Note that this sufficient condition holds in both versions of
Avoider-Enforcer games (the strict and the monotone rules). One
major disadvantage of the condition in Theorem
\ref{ReverseBox::AvWinCrt} is that $q$ does not appear in it. This
fact might indicate that, at least for large values of $q$, the
condition is far from being tight.

In this paper, as another step towards understanding
Avoider-Enforcer games we examine the mis\'ere version of the well
known Maker-Breaker game \emph{Box}, defined by Chv\'{a}tal and
Erd\H{o}s in \cite{CE}. The game Box is a $(p,q)$ Maker-Breaker
game, where the target sets (referred to as \emph{boxes}) are
disjoint. Chv\'{a}tal and Erd\H{o}s used this game as an auxiliary
game to provide Breaker with a winning strategy in the biased
connectivity game played on $E(K_n)$. They showed that for every
$\varepsilon > 0$ and for every $q\geq (1+\varepsilon)n/ \ln n$, in
a $(1,q)$ Maker-Breaker game played on $E(K_n)$, Breaker has a
strategy to isolate a vertex in Maker's graph (provided that $n$ is
large enough). Their result implies that Breaker wins various
natural games played on $E(K_n)$ such as \emph{the connectivity
game}, \emph{the perfect matching game} (where Maker is trying to
build a perfect matching) and \emph{the Hamiltonicity game} (where
Maker is trying to build a Hamilton cycle), provided that $q\geq
(1+\varepsilon)n/ \ln n$. It turns out that this result is
asymptotically tight for various games as later proved by Gebauer
and Szab\'o \cite{GS} (the connectivity game) and by Krivelevich
\cite{K} (the perfect matching and the Hamiltonicity games). Since
the paper of Chv\'{a}tal and Erd\H{o}s \cite{CE} is definitely a
cornerstone in the theory of Maker-Breaker games, it is natural to
investigate the mis\`{e}re version of the game \emph{Box}, referred
to as the \emph{mis\'ere box game}.

Let $p$ and $q$ be two positive integers. Let $b_1\leq \ldots \leq
b_n$ be a non-decreasing sequence of positive integers and let
$\cF=\{B_1,\ldots,B_n\}$ be a hypergraph such that $|B_i|=b_i$ for
every $1\leq i\leq n$ and $B_i\cap B_j=\emptyset$ for every $1\leq
i\neq j\leq n$. The mis\'ere box game $mBox(b_1, \ldots, b_n,(p,q))$
is just the $(p,q)$ Avoider-Enforcer game $\cF$ (played according to
the strict rules). If all boxes are of equal size $b_1=\ldots
=b_n=k$ (the uniform game), then we denote this game by $mBox(n
\times k,(p,q))$. Analogously, we denote by
\emph{monotone}-$mBox(b_1, \ldots, b_n,(p,q))$ and
\emph{monotone}-$mBox(n \times k,(p,q))$ the corresponding mis\'ere
box games played according to the set of monotone rules.

Note that even in this simple game the lack of monotonicity in the
strict rules is noticeable. For example, consider the
$mBox(2,2,(p,q))$ game where Avoider is the first player to move. It
is easy to verify that the case $(p,q)=(1,1)$ is Avoider's win, the
case $(p,q)=(1,2)$ Enforcer's win, and the case $(p,q)=(2,2)$ is
Avoider's win again. Therefore, this game is monotone in neither $p$
nor $q$.

Our main results are the following:

\begin{theorem} \label{ReverseBox::AvoiderWin}
Let $p,q,n$ be positive integers and let $b_1 \leq \ldots \leq b_n$.
If there exists a positive integer $k$ such that $k\leq b_1$ and
$gcd(p+q,k)
> p$, then Avoider wins the game $mBox(b_1, \ldots, b_n,(p,q))$ as a
first or a second player.
\end{theorem}

\begin{theorem} \label{ReverseBox::EnforcerWin}
Let $p,q,k$ be positive integers such that $gcd(p+q,\ell)\leq p$ for
every $1\leq \ell \leq k$. Then there exists an integer $N=N(p,q,k)$
such that for every integer $n\geq N$ and for every $b_1\leq b_2
\leq \ldots \leq b_n$ such that $b_{N} \leq k$, Enforcer has a
winning strategy in the game $mBox(b_1, \ldots, b_n,(p,q))$ as a
first or a second player.
\end{theorem}

%

Combining Theorems \ref{ReverseBox::AvoiderWin} and
\ref{ReverseBox::EnforcerWin} we get the following necessary and
sufficient condition for Enforcer's win in the uniform game,
provided that $n$ is large enough:

\begin{corollary} \label{ReverseBox::MainOldRules}
Let $p,q,k$ be three integers. Then there exists an integer
$N=N(p,q,k)$ such that for every $n\geq N$ Enforcer has a winning
strategy in the game $mBox(n \times k,(p,q))$ as a first or a second
player if and only if $gcd(p+q,\ell)\leq p$ for every $1\leq
\ell\leq k$.
\end{corollary}

Although the above theorems are about Avoider-Enforcer games played
on an edge-disjoint hypergraph, the following immediate corollary of
Theorem \ref{ReverseBox::EnforcerWin} helps us to provide a winning
strategy for Enforcer on a general (not necessarily edge-disjoint)
hypergraph.

\begin{corollary} \label{ReverseBox::HyperMatching}
Let $p,q,k$ be positive integers such that $gcd(p+q,\ell)\leq p$ for
every $1\leq \ell \leq k$. Then there exists an integer $N=N(p,q,k)$
such that for every hypergraph $\cF$, if $\cF$ contains a matching
$\cM \subseteq \cF$ which satisfies:
\begin{enumerate} [(1)]
\item $|\cM| = N$;
\item $\max\{|F| : F\in \cM \} \leq k$;
\end{enumerate}

then Enforcer has a winning strategy in the $(p,q)$ Avoider-Enforcer
game $\cF$ as a first or a second player.
\end{corollary}

We now state our results about the mis\'ere box game played
according to the set of monotone rules.

By Theorem \ref{ReverseBox::AvWinCrt} we have that for every
hypergraph $\mathcal F$ with all edges of size at least $k$, if
$|\mathcal F|< \left(1+\frac{1}{p}\right)^{k-p}$, Avoider wins the
$(p,q)$ game $\mathcal F$ as a first or a second player for every
$q$. In the following theorem we improve this result for the case
where $q \geq 2kp$ and $\mathcal F$ is an edge-disjoint hypergraph,
in particular providing a winning criterion depending on $q$.

\begin{theorem}\label{ReverseBox::AvoiderWinMonotoneRules}

Let $p,q,k,n$ be positive integers such that $k > p$, $q\geq kp$ and
$n\leq (q-p)\left(\frac{q}{kp} +1\right)^{k-p-1}$. Then for every
$k\leq b_1\leq\ldots\leq b_n$, Avoider has a winning strategy in the
game monotone-$mBox(b_1, \ldots, b_n,(p,q))$ as a first or a second
player.
\end{theorem}

\textbf{Remark:} The case $k\leq p$ (in fact $b_1\leq p$ ) is
trivial. Enforcer may fully claim all the boxes but $B_1$ and
Avoider will lose in his next move. However this might be an illegal
move for Enforcer, if there are less than $q$ elements in these
boxes. In any case, the second player --- whoever that is --- makes
at most one move and this is a simple case study.

%

\begin{theorem}\label{ReverseBox::EnforcerWinMonotoneRules}

Let $p,q$ be positive integers. For every positive integer $k$ there
exists an integer $N=N(p,q,k)$ such that for every $n\geq N$ and for
every $b_1\leq \ldots \leq b_n$ which satisfy
$\frac{1}{N}\sum_{i=1}^{N} b_i\leq k$, Enforcer has a winning
strategy in the game monotone-$mBox(b_1, \ldots, b_n,(p,q))$ as a
first or a second player.


\end{theorem}

The following immediate corollary of Theorem
\ref{ReverseBox::EnforcerWinMonotoneRules} can be used to provide a
winning strategy for Enforcer on a general hypergraph.

\begin{corollary} \label{ReverseBox::HyperMatchingMonotone}
Let $p,q$ be positive integers. For every positive integer $k$ there
exists an integer $N=N(p,q,k)$ such that for every hypergraph $\cF$,
if $\cF$ contains a matching $\cM\subseteq \cF$ which satisfies:
\begin{enumerate}[(1)]
\item $|\cM|=N$;
\item $\frac{1}{N}\sum_{F\in \cM}|F|\leq k$;
\end{enumerate}
then Enforcer has a winning strategy in the $(p,q)$ Avoider-Enforcer
game $\cF$ played according the set of monotone rules as a first or
a second player.
\end{corollary}

Analogously to the Maker-Breaker variant, the mis\'ere box game
might be useful in analyzing many other Avoider-Enforcer games which
are much more involved. Therefore, it can be helpful to estimate the
value of $N$ from Theorem~\ref{ReverseBox::EnforcerWinMonotoneRules}
for the cases $p=1$ and $q=1$.

\begin{corollary}\label{ReverseBox::N(k)_Estimate}
The following two estimates hold:
\begin{enumerate} [(i)]
\item $N(1,q,k) \leq (1+q)^k$.
\item $N(p,1,k) \leq 1 + e^{\frac kp}$.
\end{enumerate}
\end{corollary}


Next, we present two examples for which the mis\'ere box game is
used as an auxiliary game --- one for the strict rules and one for
the monotone rules.

Given positive integers $p,q$ and a fixed graph $H$, the
\emph{$H$-game} is a $(p,q)$ Avoider-Enforcer game where the board
is the edge set of a graph $G$ and the winning sets are all the
edge-sets of subgraphs of $G$ which are isomorphic to $H$. In the
following corollary we show that given a fixed graph $H$ and a large
and dense enough graph $G$, for appropriate integers $p$ and $q$,
Enforcer has a winning strategy in the $H$-game played on $E(G)$
according to the strict rules.
\begin{corollary} \label{ReverseBox::SmallGraphs}
Let $p,q,k$ be positive integers for which $gcd(p+q,\ell)\leq p$ for
every $1\leq \ell \leq k$ and let $\varepsilon>0$. Then there exists
an integer $N=N(p,q,k,\varepsilon)$ such that for every $n\geq N$
the following holds: \\Suppose that

\begin{enumerate}[(i)]
\item $H$ is a graph with $|E(H)|=k$;
\item $G$ is a graph with $|V(G)|=n\geq N$ vertices;
\item $|E(G)|\geq
\left(1-\frac{1}{\chi(H)-1}+\varepsilon\right)\frac{n^2}{2}$;
\end{enumerate}

then Enforcer has a winning strategy in the $(p,q)$ $H$-game, played
on $E(G)$ according to the strict rules.
\end{corollary}

In the following corollary we give a sufficient condition for
Avoider to avoid touching a vertex while playing according to the
set of monotone rules. This provides Avoider with a winning strategy
in various natural games, such as the \emph{the connectivity game},
avoiding a Hamilton cycle game, etc.

\begin{corollary} \label{ReverseBox::isolatingaVertex}
Let $G$ be a graph with $|V(G)|=n$ and $\Delta(G)=d<\frac n2 -1$.
Then for every $q\geq \frac{d}{\ln(n/(2d+2))}$, in the $(1,q)$
Avoider-Enforcer game played on $E(G)$ according to the set of
monotone rules, Avoider has a strategy to isolate a vertex in his
graph.
\end{corollary}

\noindent The rest of this paper is organized as follows: In
Subsection \ref{ReverseBox::subsec::notations} we introduce some
notation and terminology that will be used throughout this paper. In
Section~\ref{ReverseBox::sec::strict} we prove Theorems
\ref{ReverseBox::AvoiderWin} and \ref{ReverseBox::EnforcerWin}.  In
Section~\ref{ReverseBox::sec::Monotone} we prove Theorems
~\ref{ReverseBox::AvoiderWinMonotoneRules} and
\ref{ReverseBox::EnforcerWinMonotoneRules}, and
Corollary~\ref{ReverseBox::N(k)_Estimate}. In Section
\ref{ReverseBox::sec::applications} we prove Corollaries
\ref{ReverseBox::SmallGraphs} and
\ref{ReverseBox::isolatingaVertex}. Finally, in
Section~\ref{ReverseBox::sec::openprob} we present some concluding
remarks and open problems.

\subsection{Notation} \label{ReverseBox::subsec::notations}

The act of claiming one previously unclaimed element by one of the
players is called a \emph{step}. A \emph{move} in the strict $(p,q)$
game is a sequence of $p$ steps by Avoider, or $q$ steps by
Enforcer. Similarly, in the monotone game, each move consists of at
least $p$ or $q$ steps, respectively. A \emph{round} in the game
consists of one move of the first player, followed by one move of
the second player. When one of the players claims an element in one
of the boxes we say he \emph{touches} that box.

A box $B$ which hasn't been fully claimed yet is called a
\emph{surviving} box. A surviving box $B$ which Enforcer hasn't
touched yet is called \emph{dangerous}, otherwise it is called
\emph{safe}. An unclaimed element in a safe box is called a
\emph{safe element}. A step in which Avoider claims a safe element
is called a \emph{safe step}. A move in which Avoider makes only
safe steps is called a \emph{safe move}, otherwise it is called a
\emph{dangerous move}.

The \emph{size} of a box is the number of unclaimed elements
remained in that box. We denote the boxes by $B_1,B_2,\ldots,B_n$.
For every $1 \leq i \leq n$ we denote the size of the box $B_i$ by
$b_i$, and the average size of the first $i$ boxes by $\bar{b}_i$.
After every round we relabel the boxes so that $b_1\leq\ldots\leq
b_{n'}$, where $n'$ is the number of the surviving boxes.

\section{The strict rules} \label{ReverseBox::sec::strict}

\subsection{Avoider's win}\label{ReverseBox::subsec::Avoider}

In this subsection we prove Theorem \ref{ReverseBox::AvoiderWin}.

Throughout this subsection, let $p,q,k$ be three positive integers
and let
$$d:= gcd(p+q, k).$$ For proving Theorem
\ref{ReverseBox::AvoiderWin} we need the following two lemmas:

\begin{lemma} \label{ReverseBox::AvoiderWinLemma1}
Let $n$ be a positive integer. If $d>p$, then Avoider, as a second
player, can avoid making dangerous moves in the game $mBox(n \times
k,(p,q))$.
\end{lemma}

\begin{proof}
Assume towards a contradiction that the claim is false, and that in
his $i$th move Avoider cannot make a safe move for the first time.
Let $0\leq s <p$ be the number of safe elements on the board,
immediately before Avoider's $i$th move. Since Avoider's $j$th move
was safe for every $j<i$, it follows that all $m$ boxes which have
been touched so far during the game are safe. Therefore, exactly
$mk-s$ elements have been claimed in these boxes by both Avoider and
Enforcer. Since so far Enforcer has claimed exactly $iq$ elements
and Avoider has claimed exactly $(i-1)p$ elements, it follows that
$iq+(i-1)p=mk-s$ which implies $i(p+q)-mk=p-s$. Since $d | (p+q)$
and $d | k$, it follows that $d | (p-s)$. Recall that $0<p-s$, which
implies $d \leq p-s$. But $d>p$, a contradiction. This completes the
proof.
\end{proof}

\begin{lemma} \label{ReverseBox::AvoiderWinLemma2}

Let $n$ be a positive integer and let $0 \leq r \leq p$ be an
integer. Let $b_1=k-r$ and $b_2=\ldots=b_n=k$. If $d>p$, then
Avoider, as a second player, has a winning strategy in the game
$mBox(b_1, \ldots, b_n,(p,q))$.

\end{lemma}

\begin{proof}
Notice first that since $d=gcd(p+q,k)$, it follows that $d | k$.
Therefore, $r \leq p < d \leq k$, which implies that $b_1$ is indeed
a positive integer.

Now we describe a strategy for Avoider and then we prove it is a
winning strategy. At any point during the game, if Avoider is unable
to follow the proposed strategy then he forfeits the game. The
strategy of Avoider is as follows:

\begin{enumerate}[(i)]
\item If there are at least $p$ safe elements on the board, then Avoider
claims arbitrarily $p$ such elements.
\item Otherwise, let $s$ be the number of safe elements on the
board ($0\leq s <p$) and let $B$ be an arbitrary dangerous box.
Avoider claims all the safe elements and then he claims $p-s$ more
elements from $B$.
\end{enumerate}

We prove by induction on the number of boxes $n$ that this is indeed
a winning strategy for Avoider.

Let $n=1$. In this case, since Enforcer is the first player, he must
claim an element in the only box $B_1$. Hence, Avoider trivially
wins this game.

Assume now that $n>1$. Denote by $B$ the box which is labeled $B_1$
at the beginning of the game. Notice that by Lemma
\ref{ReverseBox::AvoiderWinLemma1}, as long as Enforcer does not
claim elements in $B$, Avoider can make safe moves, therefore he can
play according to part (i) of the proposed strategy. It follows that
if in his $i$th move Avoider has to play according to part (ii) for
the first time, then all $m$ boxes which have been touched so far by
either of the players are safe, and $B$ must be one of them.
Moreover, there must be at least one dangerous box. Notice that
immediately before Avoider's $i$th move there were $0\leq s <p$ safe
elements on the board. Therefore, we have that $iq + (i-1)p =
(m-1)k+(k-r)-s$ which implies $i(p+q)-mk=p-(s+r)$. Since $d | (p+q)$
and $d | k$, it follows that $d|(p-(s+r))$. Recall that $-p \leq
p-(s+r)\leq p$ and $d>p$, which implies that $s+r=p$. At his $i$th
move, playing according to part (ii) of the proposed strategy,
Avoider claims all $s$ safe elements on the board and $r$ more
elements from an arbitrary dangerous box.

After Avoider's $i$th move, there is exactly one box of size $k-r$
and $n-m-1$ boxes of size $k$. Since it is Enforcer's turn to move,
it follows by the induction hypothesis that Avoider has a strategy
to win this game.
\end{proof}

Now we are ready to prove Theorem \ref{ReverseBox::AvoiderWin}.

\textbf{Proof of Theorem~\ref{ReverseBox::AvoiderWin}:} First we
describe a strategy for Avoider and then prove it is a winning
strategy. At any point during the game, if Avoider is unable to
follow the proposed strategy then he forfeits the game. Avoider's
strategy is divided into the following two stages:

\textbf{Stage I:} This stage begins at the beginning of the game and
ends at the first moment during Avoider's move in which all the
dangerous boxes are of size at most $k$ and there are no safe
elements. At each step of this stage Avoider plays as follows:

\begin{enumerate}[(i)]
\item If there exists at least one element in a dangerous box of size greater than $k$ prior to this step, then Avoider claims
arbitrarily one of these elements.
\item Otherwise, Avoider claims an arbitrary safe element.
\end{enumerate}

At the end of Stage I (which may be immediate, if
$b_1=\ldots=b_n=k$), Avoider proceeds to Stage II.

\textbf{Stage II:} Let $0< r \leq p$ denote the number of Avoider's
remaining steps in his move at the moment Stage I has ended. Let $B$
be a box of size exactly $k$. Avoider claims $r$ elements from $B$.
From this point, Avoider plays according to the strategy proposed in
Lemma \ref{ReverseBox::AvoiderWinLemma2}.

It is evident that Avoider can follow Stage I of the proposed
strategy without forfeiting the game and that if the game ends at
this stage (that is, there are no more elements to claim), then
Avoider wins the game. It thus suffices to prove that Avoider also
wins even if the game ends at Stage II. Assume that in his $i$th
move Avoider proceeds to Stage II. In particular, it means that all
the surviving boxes are dangerous and each of them is of size
exactly $k$. Let $0\leq s<p$ be the number of steps Avoider can play
in his $i$th move according to Stage I of the proposed strategy and
let $r:=p-s$. At the beginning of Stage II, Avoider claims $r$
elements of one dangerous box $B$. Thus, all the boxes but $B$ are
of size exactly $k$ and $|B|=k-r$. Therefore, by Lemma
\ref{ReverseBox::AvoiderWinLemma2} we conclude that by playing
according to the proposed strategy at Stage II, Avoider wins the
game. {\hfill $\Box$\medskip\\}

\subsection{Enforcer's win}\label{ReverseBox::subsec::Enforcer}

In this subsection we prove Theorem \ref{ReverseBox::EnforcerWin}.

Let $p,q,k$ be three integers. Define:
$$ N(p,q,k):=\left\{
         \begin{array}{ll}
           (q+1)(\lceil\frac{q}{p}\rceil+3)^{k-1} & \hbox{$k \leq p$}, \\
           (2(p+q+1))^k & \hbox{$k > p$}.
         \end{array}
       \right. $$

Since $p,q$ are fixed throughout the whole game and since $k$ is the
only parameter which we change during the proof, we denote
$N(k):=N(p,q,k)$.

We prove the following theorem which trivially implies Theorem
\ref{ReverseBox::EnforcerWin}:

\begin{theorem} \label{ReverseBox::EnforcerWin1}
Let $p,q,n,t$ be four integers, where  $0\leq t\leq q$, and let $b_1
\leq b_2 \leq \ldots \leq b_n$. Assume that there exists an integer
$k$ such that the following properties hold:
\begin{enumerate}[(i)]
\item $gcd(p+q,\ell)\leq p$ for every $1\leq \ell \leq k$;
\item $n\geq N(k)$;
\item $b_{N(k)} \leq k$;
\end{enumerate}

then, Enforcer has a winning strategy in the game $mBox(b_1, \ldots,
b_n,(p,q))$ as a first player even if in his first move he claims
$t$ elements.
\end{theorem}

\textbf{Remark:} Notice that the case $t=0$ implies that Enforcer
wins the game as a second player as well.

\textbf{Proof of Theorem~\ref{ReverseBox::EnforcerWin1}:} First we
describe a strategy for Enforcer and then prove it is a winning
strategy.

\textbf{Enforcer's strategy ${\mathcal S}$:} At every step of the
game, if there are safe elements, then Enforcer claims one such
element arbitrarily. Otherwise, Enforcer claims an element in the
largest box (ties are broken arbitrarily).

It is evident that Enforcer can play according to the proposed
strategy. It thus suffices to prove that the proposed strategy is
indeed a winning strategy for Enforcer.

Before proving it we first establish the following useful lemma:

\begin{lemma} \label{ReverseBox::EnforcerWinLemma}

Let $n,\ell$ be positive integers and let $b_1\leq b_2\leq \ldots
\leq b_n=\ell$ be integers. Assume that Enforcer plays the game
$mBox(b_1, \ldots, b_n,(p,q))$ according to the strategy $\cS$.
Then, as long as the size of the largest box is $\ell$, Avoider
cannot make $\ell$ consecutive safe moves.
\end{lemma}

\begin{proof}
Assume that Enforcer is the first player to move (otherwise, after
Avoider's first move either Avoider had already lost or Enforcer is
the first player in a new game $mBox(b'_1, \ldots, b'_n,(p,q))$,
where $b'_i \leq b_i$ for all $i$. We may also assume that $b'_n =
\ell$, otherwise the claim is trivial). By definition, playing
according to the strategy $\cS$, Enforcer ensures that at any point
during the game there exists at most one safe box.

Denote $d:=gcd(p+q,\ell)$ and $t:=\frac{\ell}{d}$. Suppose that in
his $i$th move Avoider starts a succession of safe moves, all of
them in boxes of size $\ell$. We prove that he cannot make $\ell$
such moves. Let $0 < r \leq \ell$ be the number of claimed elements
in the safe box (recall that by $\cS$ there is at most one safe box)
at the beginning of round $i$. Note that the case $r=0$ means that
no elements have been claimed in the (new) safe box, which means
that actually none of the boxes is safe. This case is covered by the
case $r=\ell$ (all the elements have been claimed in the (previous)
safe box).

Express $r$ as $r=r_1d+r_2$, where $0 \leq r_1 \leq t-1$ and $0 <
r_2 \leq d$. From number theory we know that there exists an $a \in
\mathbb{Z}_\ell$ such that $a(p+q) \equiv (t-r_1)d$ (mod $\ell$). It
follows that if Avoider keeps playing safe moves, then after $a$
rounds (at the end of the $(i+a-1)$st round) there are $(r+a(p+q))$
(mod $\ell$) $\equiv (r_1d+r_2 + (t-r_1)d)$ (mod $\ell$) $= r_2$
claimed elements in the safe box. Since $0<r_2 \leq d \leq p$, and
since Avoider has made the last $p$ steps, it follows that all the
elements in this box have been claimed by Avoider. Therefore, his
$(i+a-1)$st move is not safe.

Hence, Avoider cannot make $a\leq t \leq \ell$ consecutive safe
moves. This completes the proof of the lemma.
\end{proof}

Now, by induction on $k$ we prove that the strategy ${\mathcal S}$
is indeed a winning strategy.

Assume that $k=1$. Playing according to the strategy $\cS$, Enforcer
always claims elements from a largest box. Since there are at least
$N(1) = q+1$ boxes of size $1$, we conclude that, at some point,
Avoider is forced to claim an element from a box of size $1$ and
then he loses the game.

Now, assume that $k>1$ and that for every $\ell<k$, if $b_{N(\ell)}
\leq \ell$, then ${\mathcal S}$ is indeed a winning strategy for
Enforcer in the game $mBox(b_1, \ldots, b_n,(p,q))$ even if he
claims $t$ elements in his first move for some $0\leq t \leq q$.

Notice that it suffices to prove the claim for $n=N(k)$. Indeed, for
any larger $n$, by playing according to $\cS$, in Enforcer's first
step after $n-N(k)$ boxes are fully claimed, we have that $b_{N(k)}
\leq k$ and that Enforcer has $0\leq t \leq q$ more remaining steps
to complete his move.

We prove that by playing according to ${\mathcal S}$, at some point
during the game there exists an integer $1 \leq \ell < k$ such that
at least $N(\ell)$ boxes are still dangerous and $b_{N(\ell)} \leq
\ell$. Then, by the induction hypothesis we conclude that indeed, by
playing according to ${\mathcal S}$, Enforcer wins the game.

Assume towards a contradiction that at any point during the game,
for every  $1\leq \ell <k$ we have that either there are less than
$N(\ell)$ dangerous boxes or $b_{N(\ell)}>\ell$. In particular, it
means that at the beginning of the game $b_{N(k-1)} = k$, and that
while there are still boxes of size $k$ Avoider cannot claim
elements in more than $N(k-1)$ dangerous boxes. Otherwise, we would
have $b_{N(k-1)} \leq k-1$ (since by $\cS$ Enforcer will not touch
these boxes, as they are not the largest possible). Moreover,
Avoider cannot claim two elements from more than $N(k-2)$ dangerous
boxes, otherwise we would have $b_{N(k-2)} \leq k-2$. In the same
manner we get that Avoider can make at most $\sum_{i=1}^{k-1}N(i)$
steps in dangerous boxes while there are still boxes of size $k$.

Therefore, by the time that the largest dangerous box is of size at
most $k-1$, at least $N(k)-N(k-1)$ boxes of size $k$ are fully
claimed. We distinguish between two cases:

\textbf{Case 1:} $k \leq p$. In this case, Avoider claims at most
$k-1$ safe elements and at least
  $p-k+1$ dangerous elements per move. All the dangerous boxes he
  touches become of size smaller than $k$. Therefore, in every round at most
  $q+k-1$ elements are claimed in boxes of size $k$ which are not
  dangerous after that round. It follows that
  it takes at least $\frac{(N(k)-N(k-1))k}{q+k-1}$ rounds to fully
  claim all these boxes. Hence, by the
  time that dangerous boxes of size $k$ no longer exist, the number of
  dangerous steps Avoider must have played is at least
  \begin{align*}
  &\frac{(N(k)-N(k-1))k}{q+k-1}(p-k+1) = \nonumber \\
  &\frac{((q+1)(\lceil\frac{q}{p}\rceil+3)^{k-1}-(q+1)(\lceil\frac{q}{p}\rceil+3)^{k-2})k}{q+k-1}(p-k+1)=
  \nonumber \\
  &(q+1)(\lceil\frac{q}{p}\rceil+3)^{k-2}(\lceil\frac{q}{p}\rceil+3-1)\frac{k(p-k+1)}{q+k-1}\geq
  \nonumber \\
  &N(k-1)(\frac{q}{p}+2)\frac{p}{q+p-1}> \nonumber \\
  &N(k-1)\frac{q+2p}{q+p} = \nonumber \\
  &N(k-1)(1+\frac{p}{q+p}) \nonumber
  \end{align*}
  where the first inequality follows from the fact that the quotient
  reaches minimum value at $k=p$.

  On the other hand, we have that
  $$\sum_{i=1}^{k-1}N(i) =
  (q+1)\frac{(\lceil\frac{q}{p}\rceil+3)^{k-1}-1}{(\lceil\frac{q}{p}\rceil+3)-1}<$$
  $$(q+1)(\lceil\frac{q}{p}\rceil+3)^{k-2}\frac{\lceil\frac{q}{p}\rceil+3}{\lceil\frac{q}{p}\rceil+2}
    \leq N(k-1)(1+\frac{p}{q+2p})$$
    which is clearly a contradiction.

\textbf{Case 2:} $k > p$. Since $N(k) - N(k-1) > N(k)/2$ and since
claiming all the elements in the boxes of size $k$ takes at least
$\frac{(N(k)-N(k-1))k}{p+q}$ rounds, Lemma
\ref{ReverseBox::EnforcerWinLemma} implies that Avoider must have
made at least $\frac{N(k)-N(k-1)}{(p+q)}$ dangerous moves. The
following inequality leads to a contradiction:
$$\sum_{i=1}^{k-1}N(i) \leq
\frac{2(p+q+1)((2(p+q+1))^{k-1}-1)}{2(p+q+1)-1} <$$
$$<\frac{(2(p+q+1))^{k}}{2(p+q)} = \frac{N(k)}{2(p+q)} <
\frac{N(k)-N(k-1)}{(p+q)}.$$

This completes the proof. {\hfill $\Box$\medskip\\}

\section{The monotone rules} \label{ReverseBox::sec::Monotone}

\subsection{Avoider's win} \label{ReverseBox::subsec::MonotoneAvoider}

In this section we prove Theorem
\ref{ReverseBox::AvoiderWinMonotoneRules}. In order to simplify the
proof, for every three integers $p,q,k$, we define:
$$ N(p,q,k):=\left\{
         \begin{array}{ll}
           q & \hbox{$k = p+1$}, \\
           (q-p)\left(\frac{q}{kp}+1\right)^{k-p-1} & \hbox{$k > p+1$}.
         \end{array}
       \right. $$

In fact, we can use the slightly weaker but simpler general formula
$N(p,q,k) = (q-p)\left(\frac{q}{kp}+1\right)^{k-p-1}$ for any $k
\geq p+1$, but for the purposes of the proof it will be easier to
use the above definition. Since $p,q$ are fixed and $k$ is the only
parameter we change during the proof, we denote $N(k):=N(p,q,k)$. We
show that Theorem \ref{ReverseBox::AvoiderWinMonotoneRules} holds
for every $n\leq N(k)$.

\begin{proof} First we make some assumptions to simplify the analysis. We
may assume that Avoider is the first player to move since otherwise,
after Enforcer's first move, Avoider can just claim all the safe
elements (if there are any) and pretend he is the first player in a
new game with fewer boxes. We may also assume that $b_1 = \ldots =
b_n = k$. Indeed, if some of the boxes are of size larger than $k$,
then in his first move Avoider can reduce the size of each box to
exactly $k$ and then pretend he starts a new game.


In addition, throughout the game we assume that whenever Enforcer
touches a box, he claims all the elements in this box (in this case
we simply say that Enforcer claims the box). If this is not the
case, then at the beginning of every move Avoider can claim all the
safe elements on the board and then pretend he has just started his
move. Finally, if Enforcer claims a box $B_i$ and at the end of his
move there is still a dangerous box $B_j$ such that $b_i < b_j$,
Avoider in his next move can claim $b_j - b_i$ elements from $B_j$
and pretend that Enforcer has claimed $B_j$ instead. So we may
assume that Enforcer only claims boxes of maximal size.

Now, under these assumptions, we present a strategy for Avoider and
then prove it is a winning strategy. At any point throughout the
game, if Avoider is unable to follow the proposed strategy, then he
forfeits the game.

\textbf{ Avoider's strategy $\mathcal S$:} In every move, Avoider
plays as follows:
\begin{enumerate}[(i)]
\item If there are at most $q$ dangerous boxes left, then Avoider
claims all the elements but one in each of the boxes and finishes
his move.
\item Otherwise, if there are at least $p$ boxes of maximal size, then Avoider chooses $p$ arbitrary such boxes, and from each box he claims one element.
\item Otherwise, there are $r<p$ boxes of maximal size. Avoider first claims exactly one element from each of them.
Subsequently, Avoider chooses $p$ arbitrary boxes and claims one
element from each such box.

\end{enumerate}

We prove by induction on $k$ that by playing according to $\mathcal
S$, Avoider wins the game. First, assume that $k=p+1$. In this case,
since $n\leq q$, Avoider plays according to (i) of $\mathcal S$ and
wins after Enforcer's first move. Second, let $k>p+1$, and assume
that the claim is true for all $p+1\leq \ell < k$. Assume that
Avoider follows ${\mathcal S}$ and that Enforcer follows some fixed
strategy (with the above mentioned assumptions).

Denote by \emph{Stage 1} all the rounds in the game in which only
boxes of size $k$ are being touched (that is, boxes which were of
size $k$ at the beginning of the round). If at any point during
Stage 1 the number of dangerous boxes is reduced to $q$, Avoider
plays according to $(i)$ and wins, so assume that this not the case.
Thus, at each round during Stage 1 Avoider claims exactly one
element in exactly $p$ boxes of size $k$. Enforcer then responds by
claiming at least $\lceil \frac{q}{k} \rceil$ boxes. Hence, Stage
$1$ lasts at most $$\left\lfloor \frac{n}{\lceil \frac{q}{k}
\rceil+p} \right\rfloor \leq \frac{n}{\frac{q}{k} +p}$$ rounds, in
which Avoider reduces the size of at most $p\frac{n}{\frac{q}{k}
+p}$ boxes to $k-1$. In his first move after Stage 1 Avoider touches
at most $p$ additional boxes of size $k$. Then, there are at most
$$p + p\frac{n}{\frac{q}{k} +p}\leq p+\frac{N(k)}{\frac{q}{kp}+1}$$
dangerous boxes, each of size exactly $k-1$. It remains to show that
$p+\frac{N(k)}{\frac{q}{kp}+1} \leq N(k-1)$, and then, by the
induction hypothesis, Avoider wins the game.

Indeed, for $k=p+2$ we have
$$p+\frac{N(k)}{\frac{q}{kp}+1}=p+(q-p)=q=N(k-1).$$ For $k>p+2$,
note that $$\left(\frac{q}{(k-1)p}+1\right)^{k-p-2} -
\left(\frac{q}{kp}+1\right)^{k-p-2} \geq$$$$
\left(\frac{q}{(k-1)p}+1\right) - \left(\frac{q}{kp}+1\right) =
\frac {q}{k(k-1)p} \geq \frac {1}{k-1},$$

where the last inequality follows from the fact that $q \geq kp$.
Using this fact again and the above calculation, we get
$$p+\frac{N(k)}{\frac{q}{kp}+1} = p +
(q-p)\left(\frac{q}{kp}+1\right)^{k-p-2} \leq
(q-p)\left(\frac{q}{(k-1)p}+1\right)^{k-p-2} = N(k-1)$$ as required.
\end{proof}

\subsection{Enforcer's win}

In this subsection we prove Theorem
\ref{ReverseBox::EnforcerWinMonotoneRules}.

\textbf{Proof of Theorem
\ref{ReverseBox::EnforcerWinMonotoneRules}:} Let $p,q,k$ be three
positive integers, define:
$$N(p,q,k)=\left\{
  \begin{array}{ll}
    q+1 & \hbox{$k\leq p$}, \\
    q+1+\left\lceil q/k \right\rceil & \hbox{$k=p+1$},\\
    \left\lceil\frac 1p N(p,q,k-1)\right\rceil\left(p+\left\lceil \frac{q}{k} \right\rceil\right) &
    \hbox{$k>p+1$}.
  \end{array}
\right.$$

Similarly to subsection~\ref{ReverseBox::subsec::MonotoneAvoider},
since $p$ and $q$ are fixed throughout the game and $k$ is the only
parameter we change during the proof, we denote $N(k):=N(p,q,k)$.

Let $n\geq N(k)$ be an integer and let $b_1\leq \ldots \leq b_n$ be
integers such that $\overline{b}_{N(k)} \leq k$. We prove that
Enforcer has a winning strategy in the game $mBox(b_1, \ldots,
b_n,(p,q))$. Clearly, it suffices to deal with the case where
Enforcer is the first player, since Avoider's move can only decrease
$\overline{b}_{N(k)}$.

First we describe a strategy for Enforcer and then prove it is a
winning strategy.

\textbf{Enforcer's strategy $\mathcal S$:} At any point during the
game if Enforcer is unable to follow the proposed strategy then he
forfeits the game. Enforcer plays each move as follows:
\begin{enumerate} [(i)]
\item If $b_1\leq p$, then Enforcer fully claims all the boxes but
$B_1$, and finishes his move.
\item If there exists an integer $\ell \leq k$ such that at
least $N(\ell)$ boxes are still dangerous and
$\overline{b}_{N(\ell)} \leq \ell$, then for the minimal such
$\ell$, Enforcer fully claims all the boxes $B_i$, for all $i >
N(\ell)$. Then, he pretends he starts a new move and proceeds to
(iii).
\item Let $m$ be the minimal integer for which the largest $m$ boxes
contain at least $q$ elements. In his move, Enforcer fully claims
the largest $m$ boxes.
\end{enumerate}

Now, we prove that Enforcer can follow the proposed strategy without
forfeiting the game and that this is indeed a winning strategy for
him.

Assume that Enforcer plays the game against some fixed strategy of
Avoider.

If $k\leq p$, then since there are at least $q+1$ boxes, it follows
that in his first move, Enforcer can claim all the elements in all
the boxes except of $B_1$. In his next move, Avoider must claim all
the elements of $B_1$ and thus loses the game.

Now we prove the theorem for every $k\geq p+1$ by induction on $k$.
We may assume for simplicity that $n=N(k)$, since otherwise Enforcer
fully claims all the boxes $B_{N(k)+1},\ldots,B_n$ in his first move
anyway, according to $\cS$.

Assume $k=p+1$. If $b_1\leq p$, then since there are at least
$q+1+\lceil q/k \rceil \geq q+1$ boxes, Enforcer wins the game in a
similar way to the case $k\leq p$. Otherwise, all the boxes must be
of size $p+1$. Playing according to $\cS$, Enforcer fully claims
$\lceil\frac{q}{p+1}\rceil$ boxes in his first move, leaving $q+1$
dangerous boxes. Then, after Avoider's first move we must have
$b_1\leq p$ and once again, Enforcer wins the game.

Assume now that $k>p+1$ and that $\mathcal S$ is a winning strategy
for Enforcer for every $\ell<k$, $n\geq N(\ell)$ and for every
$b_1\leq \ldots \leq b_n$ such that $\overline{b}_{N(\ell)} \leq
\ell$. Notice that if at any move during the game Enforcer plays
according to part (ii) of $\cS$ for some $\ell<k$, then by the
induction hypothesis he wins the game. Clearly, if at some move he
plays according to part (i) of $\cS$, he wins immediately.
Therefore, we can assume that Enforcer plays only according to part
(iii) of $\cS$. For simplicity, denote $N_p(k):=p\left\lceil
N(k)/p\right\rceil$.

Notice that Enforcer can play entirely in boxes of size at least $k$
for his first $\left\lceil\frac{1}{p}N(k-1)\right\rceil$ moves
(fully claiming at most $\left\lceil \frac{q}{k} \right\rceil$ of
them per move). Otherwise, by the time that all the surviving boxes
are of size at most $k-1$, there are more than $N(k)- \left\lceil
\frac{q}{k} \right\rceil \left\lceil\frac{1}{p}N(k-1)\right\rceil =
N_p(k-1) \geq N(k-1)$ of them, which will lead Enforcer to play
according to part (ii) of $\cS$, in contradiction to our assumption.
For the same reason we may conclude that Enforcer has claimed
exactly $\left\lceil \frac{q}{k} \right\rceil$ boxes per move (and
no less).

Let us now examine the board after $\left\lceil
\frac{1}{p}N(k-1)\right\rceil $ rounds. There are exactly
$\left\lceil \frac{q}{k} \right\rceil\left\lceil
\frac{1}{p}N(k-1)\right\rceil$ boxes that Enforcer has fully claimed
and $N_p(k-1)$ surviving boxes. Avoider must have claimed at least
$p\left\lceil\frac{1}{p}N(k-1)\right\rceil = N_p(k-1)$ elements
during these rounds. Suppose that $t$ of them were in boxes that
were later claimed by Enforcer. Since Enforcer has only touched
boxes of size at least $k$ so far, it follows that at the beginning
of the game there were at least $\left\lceil \frac{q}{k}
\right\rceil \left\lceil\frac{1}{p}N(k-1)\right\rceil k + t$
elements in the boxes that Enforcer has fully claimed, and since the
total number of elements at the beginning of the game was at most
$kN(k)$ it follows that there were at most $kN(k)-k\left\lceil
\frac{q}{k} \right\rceil \left\lceil\frac{1}{p}N(k-1)\right\rceil -
t = kN_p(k-1) - t$ elements in the surviving boxes. Since Avoider
has claimed at least $N_p(k-1) - t$ elements in the surviving boxes,
there are at most $(k-1)N_p(k-1)$ unclaimed elements in them now,
i.e. $\bar{b}_{N_p(k-1)} \leq k-1$ and in particular
$\bar{b}_{N(k-1)} \leq k-1$, and so by the induction hypothesis
Enforcer wins the game. {\hfill $\Box$\medskip\\}

\textbf{Proof of Corollary \ref{ReverseBox::N(k)_Estimate}}:

For $(i)$, note that by the definition of $N(k)$ in the proof of
Theorem~\ref{ReverseBox::EnforcerWinMonotoneRules}, for $p=1$ we
obtain $$N(k) \leq \left\lceil\frac 1p
{N(k-1)}\right\rceil\left(p+\left\lceil \frac{q}{k}
\right\rceil\right) = N(k-1)\left(1+\left\lceil\frac qk
\right\rceil\right) \leq N(k-1)(1+q).$$ Therefore, $N(k) \leq
(1+q)^k$.

For $(ii)$, we estimate $N(k)$ in the following way. Denote by $n$
the number of boxes at the beginning of the game and by $\phi(i)$
the average size of the dangerous boxes just before the beginning of
the $i$th round. In each round Enforcer claims the largest box so he
does not increase the average size of the dangerous boxes. Avoider
then claims at least $p$ elements in the remaining $n-i$ boxes.
Therefore, $\phi(i+1) \leq \phi(i)-\frac{p}{n-i}$ for all $1 < i
\leq n$. We know that $\phi(1) \leq k$ and notice that if $\phi(n) <
1$ it means that Enforcer has won the game. We have that:

$\phi(n) \leq \phi(n-1) -p \leq \phi(n-2)-p(\frac{1}{2}+1) \leq
\ldots \leq \phi(1) -p(\frac{1}{n-1} + \ldots + 1) \leq
k-p\ln(n-1)$.

So if $n>1+e^{\frac{k-1}{p}}$ Enforcer wins the game.{\hfill
$\Box$\medskip\\}

\section{Some applications} \label{ReverseBox::sec::applications}

In this section we prove Corollaries \ref{ReverseBox::SmallGraphs}
and \ref{ReverseBox::isolatingaVertex}.

\textbf{Proof of Corollary \ref{ReverseBox::SmallGraphs}:} Let
$\mathcal F$ be the hypergraph of the game. By the well known
Counting Lemma (see e.g, Theorem 2.8 at \cite{KSz}), we conclude
that $\mathcal F$ contains a matching of size $\Theta(n^2)$ (since
$G$ contains $\Theta\left(n^{|V(H)|}\right)$ distinct copies of
$H$). Now, applying Corollary \ref{ReverseBox::HyperMatching} we get
that for a large enough $n$ (compared to $p,q$ and $k=|E(H)|$),
Enforcer wins the $(p,q)$ game played on $\mathcal F$ and hence wins
the $H$-game. {\hfill $\Box$\medskip\\}

\textbf{Proof of Corollary \ref{ReverseBox::isolatingaVertex}:}
Since $\Delta(G)=d$ we can find (greedily) an independent subset
$S\subseteq V$ of size $|S|=s \geq n/(d+1)$. Denote by $b_1 \leq
\ldots \leq b_s \leq d$ the degrees in $G$ of the vertices in $S$.
Clearly, if $b_1=0$ Avoider wins the game no matter how he plays, so
assume $1 \leq b_1$. Assume for simplicity that Avoider is the first
player to move (since otherwise we can remove the edges that
Enforcer has claimed in his first move from $G$ and Avoider can
pretend he is the first player in a game on this new graph.
Obviously the set $S$ is still independent).

Avoider's strategy goes as follows:
\begin{enumerate} [(1)]
\item In his first move, Avoider claims all the edges $e\in E(G)$ for
which $e\cap S=\emptyset$.
\item From now on, Avoider pretends he is BoxEnforcer in the game
monotone-$mBox(b_1, \ldots ,b_s, (q,1))$, where the boxes are the
stars with centers in $S$, and enforces BoxAvoider (which is the
real Enforcer in the original game) to claim all the edges which
touch some vertex from $S$.
\end{enumerate}

It is evident that if Avoider can follow the proposed strategy then
he wins the game. It thus suffices to prove that he can win as
BoxEnforcer the game monotone-$mBox(b_1, \ldots ,b_s,(q,1))$ for
every $b_1 \leq \ldots \leq b_s \leq d$. By using the estimate from
Corollary~\ref{ReverseBox::N(k)_Estimate} and the following
calculation:
$$ N(q,1,d)\leq 1+e^{\frac{d-1}{q}} \leq 2e^{\frac{d}{q}} \leq n/(d+1) \leq s$$

(where the third inequality holds since $q\geq
\frac{d}{\ln(n/(2d+2))}$), we get the desired result.{\hfill
$\Box$\medskip\\}

\section{Concluding remarks} \label{ReverseBox::sec::openprob}

Avoider-Enforcer games are more difficult to analyze than
Maker-Breaker games and much less is known about them. In this paper
we examined Avoider-Enforcer games which are played on edge-disjoint
hypergraphs. We also showed that the mis\'ere box game is useful
when one wants to provide Enforcer with a winning strategy in a game
played on a general hypergraph with a large matching. In general,
our arguments do not help Avoider to win on a hypergraph which is
not edge-disjoint. However, in some cases Avoider can pretend he is
playing another game as Enforcer in order to achieve his goals, and
then one can use our arguments as we showed in Corollary
\ref{ReverseBox::isolatingaVertex}.

We believe that it is natural and interesting to investigate
Avoider-Enforcer games played on general hypergraphs. As a first
step, we suggest to consider Avoider-Enforcer games played on almost
disjoint hypergraphs (where every two hyperedges intersect in at
most one vertex), or on hypergraphs with bounded maximum degree.
Even for these relatively simple hypergraph classes not much is
known.

\medskip
\medskip

\textbf{Acknowledgment:} We would like to thank the anonymous
referees for their helpful comments.

\end{document}